\documentclass{amsart}
\usepackage{amsmath}
\usepackage{amssymb}
\usepackage{mathrsfs}
\usepackage{pgf,tikz}
\usepackage{ulem}
\usetikzlibrary{shadings}
\usetikzlibrary{arrows,automata}
\usepackage{graphicx}
\usepackage{systeme,mathtools}
\usepackage{amsthm,enumitem}
\usepackage{systeme}
\usepackage{eucal}
\usepackage{marginnote}
\usepackage{blindtext}
\usepackage{hyperref}
\input xy
\xyoption{all}

\newtheorem{theorem}{Theorem}[section]

\newtheorem{lemma}[theorem]{Lemma}
\newtheorem{corollary}[theorem]{Corollary}
\newtheorem{proposition}[theorem]{Proposition}
\theoremstyle{definition}
\newtheorem{definition}[theorem]{Definition}
\newtheorem{remark}[theorem]{Remark}
\newtheorem{example}{Example}

\begin{document}

	\title[On free subgroups in Leavitt path algebras]{On free subgroups in Leavitt path algebras}
		  
	
	\maketitle
	
	\begin{center} {BUI XUAN HAI\footnote{Corresponding author}}
	\end{center}
    \begin{center}	
    	\tiny{(1) Faculty of Mathematics and Computer Science, University of Science, Ho Chi Minh City, Vietnam\\(2) Vietnam National University, Ho Chi Minh City, Vietnam\\
    		e-mail:  bxhai@hcmus.edu.vn\\
    	ORCID: 0000-0002-4208-7883} 
    \end{center} 
\begin{center} {HUYNH VIET KHANH}
\end{center}
\begin{center}	
	\tiny{Department of Mathematics and Informatics, HCMC University of Education,\\ 280 An Duong Vuong Str., Dist. 5, Ho Chi Minh City, Vietnam\\ e-mail: khanhhv@hcmue.edu.vn\\
	ORCID: 0000-0002-2109-4611} 
\end{center}
\bigskip

	\begin{abstract} 
		Let $E$ be a graph and $K$ a field. In this paper we prove that the multiplicative group of a unital noncommutative Leavitt path algebra $L_K(E)$ contains  non-cyclic free subgroups provided $K$ is of characteristic $0$. Further, we provide a description of the generators of  such free subgroups in term of the graph $E$. 
	\end{abstract}

\section{Introduction and preliminaries}

Let $R$ be an associative ring with identity $1\ne 0$. The set $R^\times$ of all invertible elements in $R$ with the multiplicative operation  constitutes a group called  the \textit{multiplicative group} of $R$. The existence of non-cyclic free subgroups in $R^\times$ is one of the problems which have attracted the attention of many authors for the last decades. Historically, it comes from the work \cite{Pa_lichman_1977}, where Lichtman conjectured the existence of non-cyclic free subgroups in the multiplicative group of an arbitrary noncommutative division ring. Later, Gon\c calvez and Mandel \cite{Pa_Goncalves-Mandel_1986} posed the more general conjecture that any subnormal subgroup in the multiplicative group of a division ring contains a non-cyclic free subgroup.  Although there are several  publications giving the affirmative answer to this conjecture in various particular cases (see e.g. \cite{Pa_Chiba_1996}, \cite{Pa_danh-deo_2023}, \cite{Pa_Goncalves-Mandel_1986}, \cite{Pa_Goncalvez_1984} \cite{Pa_Hai-Ngoc_2013}, \cite{Pa_Ngoc-Bien-Hai}, \cite{Pa_reichstein_1995}, and references therein), at the present time, it remains still unsolved in general. Our main purpose in the present paper is to prove that the multiplicative group of any unital noncommutative Leavitt path algebra over a field of characteristic $0$ always contains a non-cyclic free subgroup. Further, we shall describe the generators for such a subgroup. For the reader's convenience, in this section, let us recall some concepts and symbols we use throughout this paper regarding Leavitt path algebras over fields. For more details we refer to the book \cite{Bo_abrams_2017}.

A \textit{directed graph} is a graph, denoted by $E=(E^0, E^1, r, s)$, consisting of two sets $E^0$ and $E^1$ together with maps $r, s: E^1\to E^0$. The elements of $E^0$ and $E^1$ are called \textit{vertices} and \textit{edges} of $E$ respectively. For $e\in E^1$, we say that $s(e)$ is the \textit{source} and $r(e)$ is the \textit{range} of $e$. In this paper, the word “graph” will always mean “directed graph”. Moreover, concerning a graph as above, we often write $E$ instead of  $E=(E^0, E^1, r, s)$. 

Let $E$ be a graph, $v\in E^0$ and $e\in E^1$. We say that $v$ \textit{emits} $e$ if $s(e)=v$. A vertex $v$ is a \textit{sink} if it emits no edges, while it is an \textit{infinite emitter }if it emits infinitely many edges. A vertex $v$ is said to be \textit{regular} if it is neither a sink nor an infinite emitter. The set of all regular vertices in a graph $E$ is denoted by ${\rm Reg}(E)$. A \textit{finite path} $\mu$ of length $\ell(\mu):=n\ge 1$ is a finite sequence of edges $\mu=e_1e_2\cdots e_n$ with $r(e_i) =s(e_{i+1})$ for all $1\leq i\leq n-1$. We set $s(\mu):=s(e_1)$ and $r(\mu):=r(e_n)$. In addition, every vertex can be considered as a path of length $0$. The set of all finite paths (including all paths of length $0$) in $E$ is denoted by ${\rm Path}(E)$. 
%The set of all finite paths in $E$ is denoted by ${\rm Path}(E)$. 

Let us define the set $\left( E^1\right)^*=\{e^*| e\in E^1\}$. Any element of $\left( E^1\right)^*$ is called a \textit{ghost edge} of $E$. Also, we set $s(e^*):=r(e)$ and $r(e^*):=s(e)$. 
%%%%%%%%%%%%%%%%%%%%%

\begin{definition}[Leavitt path algebra]\label{LPA}
	Let $E$ be an graph and $K$ a field. The \textit{Leavitt path algebra of $E$ over $K$}, denoted by $L_K(E)$, is the free associative $K$-algebra generated by the set $E^0\cup E^1\cup (E^1)^*$, subject to the following relations:
	\begin{enumerate}[]
		\item[(V)] $vv'=\delta_{v,v'}v$ for all $v,v'\in E^0$, where $\delta_{ij}$ is the Kronecker symbol,
		\item[(E1)] $s(e)e=er(e)=e$ for all $e\in E^1$,
		\item[(E2)] $r(e)e^*=e^*s(e)=e^*$ for all $e\in E^1$,
		\item[(CK1)] $e^*f=\delta_{e,f}r(e)$ for all $e,f\in E^1$, and
		\item[(CK2)] $v=\sum_{\{e\in E^1|s(e)=v\}}ee^*$ for every $v\in{\rm Reg}(E)$.
	\end{enumerate}
\end{definition}
It is known (see \cite[Lemma 1.2.12]{Bo_abrams_2017}) that $L_K(E)$ is a unital algebra if and only if $E^0$ is finite. In this case, the identity in $L_K(E)$ is
$$1_{L_K(E)}=\sum_{v\in E^0}v.$$

From now on, we consider only Leavitt path algebras of graphs with a finite number of vertices, ensuring they are unital. It follows from Condition (E1) that if $E $ has an edge $e$ with $s(e)\ne r(e)$, then 
$$s(e)e=e\ne 0=er(e)s(e)=es(e),$$ 
and so $L_K(E)$ is noncommutative. 

It is clear from Condition (V) in Definition \ref{LPA} that $L_K(E)$ cannot be a division ring unless $|E^0|=1$. Recall also that   $L_K(E)$ is semiprimitive in view of \cite[Proposition 2.3.2]{Bo_abrams_2017}. In this paper, we shall prove the surprising fact that if  $K$ is a field of characteristic $0$ then the multiplicative group $L_K(E)^\times$ of a Leavitt path algebra $L_K(E)$ always contains a non-cyclic free subgroup provided that it is noncommutative. 

%%%%%%%%%%%%%%%%%%%%%%%

If  $\mu=e_1\cdots e_n\in {\rm Path}(E)$, then we call the element $\mu^*:=e_n^*\cdots e_2^*e_1^*$ of $L_K(E)$ a \textit{ghost path}. The set of all vertices on a path $\mu$ is denoted by $\mu^0$. If $\ell(\mu)\ge 1$ and $v=s(\mu)=r(\mu)$, then we say that $\mu$ is a \textit{closed path based at} $v$. If, moreover, $s(e_j)\ne v$ for every $j>1$, then we call $\mu$ a \textit{closed simple path based at $v$}. 

If $\mu=e_1\cdots e_n$ is  a closed path based at $v$ and $s(e_i)\ne s(e_j)$ for every $i\ne j$,
then $\mu$ is called a \textit{cycle based at} $v$. If $\mu=e_1\cdots e_n$ is a cycle based at $v$, then for each $1\le i\le n$, the path $\mu_i=e_ie_{i+1}\cdots e_ne_1\cdots e_{i-1}$ is a cycle based at $s(e_i)$. We call the collection of cycles $\{\mu_i\}$ based at $s(e_i)$ the \textit{cycle of} $\mu$. A \textit{cycle} $c$ is a set of paths consisting of the cycle of $\mu$ for some cycle $\mu$ based at a vertex $v$. The \textit{length of a cycle} $c$ is the length of any path in $c$. In particular, a cycle of length $1$ is called a \textit{loop}. A cycle $c$ is  an \textit{exclusive cycle} if it is disjoint with every other cycle.

An \textit{exit} for a path $\mu=e_1\cdots e_n$ is an edge $e$ such that $s(e)=s(e_i)$ for some $i$ and $e\ne e_i$. A graph $E$ is said to satisfy the \textit{Condition $\mathrm{(L)}$} if every cycle in $E$ has an exit. 

For $u,v\in E^0$, if there exists a path $\mu$ such that  $u=s(\mu)$ and $v=r(\mu)$, then we write $u\geq v$. 

A subset $T$ of $E^0$ is said to have the \textit{countable separation property} if there exists a countable  set $C$ of $E^0$ with the property that for each $u\in T$ there is a $v\in C$ such that $u\geq v$.

For each $v\in E^0$ and $\mu\in{\rm Path}(E)$, we define $ M(v)=\{w\in E^0\;|\; w\geq v\}$ and $ M(\mu)=\{w\in E^0\;|\; w\geq v \text{ for some } v\in \mu^0\}$.

Let $H$ be a subset of $E^0$. We say that $H$ is \textit{hereditary}  if whenever $u\in H$ and $u\geq v$ for some vertex $v$, then $v\in H$; and $H$ is \textit{saturated} if, for any regular vertex $v$, $r(s^{-1}(v))\subseteq H$ implies $v\in H$. Let $H$ be a hereditary and saturated subset of $E^0$. A subset $M\subseteq H$ is said to satisfy MT-3 condition if for any $u, v\in M$ there exists $w\in M$ such that $u\geq w$ and $v\geq w$.

Let $H$ be a hereditary and saturated subset of $E^0$. A vertex $w$ is called a \textit{breaking vertex} of $H$ if $w\in E^0\backslash H$ is an infinite emitter such that $1\leq |s^{-1}(w)\cap r^{-1}(E^0\backslash H)|<\infty$. The set of all breaking vertices of $H$ is denoted by $B_H$. For each $w\in B_H$, $w^H$ denotes the element $w-\sum_{\{\substack{s(e)=w,\;r(e)\not\in H}\}}ee^*$. For a given hereditary and saturated subset $H$ of $E^0$  and a subset $S\subseteq B_H$, we call $(H,S)$ an \textit{admissible pair}, and denote by $I(H, S)$ the ideal of $L_K(E)$ generated by the set  $H\cup \{v^H: v\in S\}$. 
\begin{definition}[Quotient graph] \label{def:1.1}
	For each admissible pair $(H,S)$, the \textit{quotient graph} $E/(H,S)$ is defined as follows:
	$$ 
	\begin{aligned}
			& (E/(H,S))^0= (E^0\backslash H) \cup \{v': v\in B_H\backslash S\},\\
			& (E/(H,S))^1=\{e\in E^1: r(e)\not\in H\}\cup \{e': e\in E^1, r(e)\in B_H\backslash S\},
		\end{aligned}
	$$
	and $r$ and $s$ are extended to $(E/(H,S))^1$ by setting $s(e')=s(e)$ and $r(e')=r(e)'$. 
\end{definition}

The following remark will be employed frequently throughout the remaining part of the paper.
\begin{remark}\label{rem:1.2}
	It was shown in \cite[Theorems 2.4.15,  2.5.8]{Bo_abrams_2017} that the graded ideals of $L_K(E)$ are precisely the ideals of the form $I(H,S)$ for some admissible pair $(H,S)$. Moreover, the mapping $\varphi: L_K(E)\to L_K(E/(H,S))$, where $E/(H,S)$ is the quotient graph of $E$ by $(H,S)$, defined by
	\begin{align*}
			&	\varphi(v) =    \begin{cases} 
					v+v'        & \text{ if } v\not\in H \text{ and } v\in B_H\backslash S \\
					v             & \text{ if } v\not\in H \text{ and } v\not\in B_H\backslash S \\
					0             & \text{ if } v\in H
				\end{cases},\\
			&	\varphi(e) =    \begin{cases} 
					e+e'        & \text{ if } r(e)\not\in H \text{ and } r(e)\in B_H\backslash S \\
					e             &  \text{ if } r(e)\not\in H \text{ and } r(e)\not\in B_H\backslash S \\
					0             & \text{ if } r(e)\in H
				\end{cases},  \\                     					
			&		\varphi(e^*) =    \begin{cases} 
					e^*+(e')^*      & \text{if } \text{ if } r(e)\not\in H \text{ and } r(e)\in B_H\backslash S \\
					e^*        			&  \text{ if } r(e)\not\in H \text{ and } r(e)\not\in B_H\backslash S \\
					0        			  & \text{ if } r(e)\in H
				\end{cases}
		\end{align*}
	is an epimorphism with $\ker\varphi = I(H,S)$.
\end{remark}

Let $H$ be a hereditary and saturated subset of $E^0$, $c$ a cycle without exits in $E/(H,B_H)$ based at $v\in E^0\backslash H$, and $f(x)$ a polynomial in $K[x,x^{-1}]$. We denote by $I(H,B_H,f(c))$ the ideal of $L_K(E)$ generated by $I(H,B_H)$ and $f(c)$. Here $f(c)$ is the element of $L_K(E)$ obtained by substituting $x$ by $c$, $x^{-1}$ by $c^*$ and the constant term $a_0$ by $a_0v$ in the expression of $f(x)$ as a polynomial in $x, x^{-1}$.

\section{Simple modules over Leavitt path algebras}\label{simple modules}

In this section we briefly present some properties we need of classes of simple modules  over Leavitt path algebras introduced by Chen, Ara and Rangaswamy in   \cite{Pa_ara_rangas_2014} and \cite{Pa_chen_2015}.

(A): \textit{Chen's module $V_{[\mu]}$ defined by an infinite path $\mu$}:
Let $\mu=e_1e_2\cdots e_n\cdots$ be an infinite path, and $n$ a positive integer. In  \cite{Pa_chen_2015}, Chen defines $\tau_{\leq n}(\mu):=e_1\cdots e_n$ and $\tau_{> n}(\mu):=e_{n+1}e_{n+2}\cdots$. In addition, we set $\tau_{\leq 0}(\mu):=s(\mu)$ and $\tau_{> 0}(\mu):=\mu$. Two infinite paths $\mu$ and $\eta$ are said to be \textit{tail-equivalent} if there exist positive integers $m,n$ such that $\tau_{>m}(\mu)=\tau_{>n}(\eta)$. A simple checking shows that this is an equivalence relation. For an infinite path $\mu$, the symbol $[\mu]$ stands for the equivalence class of all paths which are  tail-equivalent to $\mu$. If $\mu=g^{\infty}:=ggg\cdots$, where $g$ is some  closed path, then we say that $\mu$ is \textit{rational}. Following Ara and Rangaswamy \cite{Pa_ara_rangas_2014}, if an infinite path $\mu$ is tail-equivalent to the rational path $c^\infty$, where $c$ is a cycle in $E$, then we say that $\mu$ \textit{ends in a cycle.} 

For an infinite path $\mu$, Chen defines 
$$
V_{[\mu]}=\bigoplus_{\eta\in[\mu]}K\eta,
$$ 
a $K$-vector space which has a $K$-basis $\{\eta:\eta\in [\mu]\}$. It was shown by Chen in \cite{Pa_chen_2015} that the following rules make $V_{[\mu]}$ into a simple left $L_K(E)$-module.

\begin{enumerate}
	\item[] $v\cdot \eta=\eta$ or $0$ according to $v=s(\eta)$ or not;
	\item[] $e\cdot \eta =e\eta$ or $0$ according to $r(e)=s(\eta)$ or not;
	\item[] $e^*\cdot \eta=\tau_{>1}(\eta)$ or $0$ according to $\eta=e\eta'$ or not.
\end{enumerate}

(B): \textit{Chen's module $\mathbf{N}_w$ defined by a sink $w$}: Let $w$ be a sink in a graph $E$. Following Chen \cite{Pa_chen_2015}, we denote by $\mathbf{N}_w$ the $K$-vector space whose basis consists of all finite paths ending in $w$; that is,
$$
\mathbf{N}_w=\bigoplus_{\substack{\eta\in{\rm Path}(E),\\ r(\eta)=w}}K\eta.
$$
Then, $\mathbf{N}_w$ becomes a simple left $L_K(E)$-module by defining an action on its basis elements in the same way we have done for $V_{[\mu]}$ (with the addition of $e^*\cdot w=0$ for all $e\in E^1$).  From  \cite[Theorem 3.3(1) and Theorem 3.7(1)]{Pa_chen_2015}, we can state the following result.
\begin{lemma}\label{lemma_2.1}
	Let $\mu$ be an infinite path, and $w$ a sink in $E$. Let $V_{[\mu]}$ and $\mathbf{N}_w$ be the $L_K(E)$-modules defined above. Then, 
	$${\rm End}_{L_K(E)}(V_{[\mu]})\cong K \hbox{ and } \;\;{\rm End}_{L_K(E)}(\mathbf{N}_w)\cong K.$$
\end{lemma}

We note that the notations $V_{[\mu]}$ and $\mathbf{N}_w$ were actually used in \cite{Pa_ara_rangas_2014} to indicate respectively the modules $\mathcal{F}_{[\mu]}$ and $\mathcal{N}_w$, which were originally defined in \cite{Pa_chen_2015} by Chen.

(C): \textit{Ara-Rangaswamy twisted module $V_{[\mu]}^f$}: Let $A$ be an algebra over a field $K$ and $M$ an $A$-module. Given an automorphism $\sigma$ of $K$-algebra $A$, we can define the twisted $A$-module $M^\sigma$ as follows.
\begin{enumerate}
	\item[] $M^\sigma=M$ as a vector space over $K$,
	\item[] $a\cdot m^\sigma=(\sigma(a)m)^\sigma$.
\end{enumerate}
(Here, for an element $m\in M$, the symbol $m^{\sigma}$ stands for the corresponding element in $M^\sigma$.) One can easily check that ${\rm End}_A(M) = {\rm End}_A(M^\sigma)$. Moreover, the twisted module $M^\sigma$ is simple if and only if $M$ is simple. In \cite{Pa_chen_2015}, Chen has studied intensively the twisted modules $V_{[\mu]}^{\mathbf{a}}$ and $\mathbf{N}_w^{\mathbf{a}}$ of $V_{[\mu]}$ and $\mathbf{N}_w$ respectively, where ${\mathbf{a}}$ is an appropriate automorphism of $L_K(E)$.\\

As a modification of Chen's construction, in \cite{Pa_ara_rangas_2014}, Ara and Rangaswamy introduced the new simple $L_K(E)$-module $V_{[\mu]}^f$, which can be described briefly as follows. Let $f(x)=1+a_1x+\dots+a_nx^n$, $n\geq1$, be an irreducible polynomial in $K[x,x^{-1}]$ and let $c=e_1e_2\cdots e_m$ be an exclusive cycle in $E$. Set $\mu=c^\infty$. Let $K'=K[x,x^{-1}]/(f(x))$, which is a field extension of $K$. Then, we can form the Leavitt path algebra $L_{K'}(E)$ over the field $K'$. Let $\bar{x}$ be the image of $x$ in $K'$. It is clear that $\bar{x}\ne 0$ in $K'$, and so $\bar{x}$ is invertible in $K'$. Therefore, we can define a map $\sigma: L_{K'}(E)\to L_{K'}(E)$ which sends 

\begin{enumerate}
	\item[] $v\to v$ for $v\in E^0$,
	\item[] $e\to e$ and $e^*\to e^*$ for $e\in E^1$ with $e\ne e_1$,
	\item[] $e_1\to \bar{x}e_1$ and $e_1^*\to \bar{x}^{-1}e_1^*$.
\end{enumerate}
It is a simple matter to check that $\sigma$ is an automorphism of $L_{K'}(E)$, and so we can define the twisted $L_{K'}(E)$-module $V_{[\mu]}^\sigma$ of the $L_{K'}(E)$-module $V_{[\mu]}$. In \cite{Pa_ara_rangas_2014}, Ara and Rangaswamy denote by $V_{[\mu]}^f$ the $L_K(E)$-module obtained by restricting scalars on $V_{[\mu]}^\sigma$ from $L_{K'}(E)$ to $L_K(E)$. Also, it was proved that $V_{[\mu]}^f$ is a simple $L_K(E)$-module. Now, for further using, we need to calculate ${\rm End}_{L_K(E)}(V_{[\mu]}^f)$. The next two lemmas are introduced as key steps toward proving the main result (Proposition~\ref{proposition_2.4}) of this section.

\begin{lemma}\label{lemma_2.2}
	Let $c=e_1e_2\cdots e_m$ be an exclusive cycle in a graph $E$ and $\mu=c^\infty$. Let $f(x)=1+a_1x+\dots+a_nx^n$, $n\geq1$, be an irreducible polynomial in $K[x,x^{-1}]$ and let $K'=K[x,x^{-1}]/(f(x))$. Let $V_{[\mu]}^f$ be the simple $L_K(E)$-module defined as above. Then, 
	$${\rm End}_{L_K(E)}(V_{[\mu]}^f)\cong K'.$$
\end{lemma}
\begin{proof}
	We follow the proofs of \cite[Theorem 3.3(i)]{Pa_chen_2015} and \cite[Lemma 3.3]{Pa_ara_rangas_2014}. Let $p\in [\mu]$ and write 
	$$
	\varphi(p)=\sum_{i=1}^{m}\lambda_i p_i,
	$$
	where $\lambda_i\in K'\backslash\{0\}$ and $p_i$ distinct paths in $[\mu]$. Then, we can uniquely write $p_i=p'_i\mu$, where $p'_i$ is a finite path (possibly of length $0$), which does not involve $e_1$. We claim that $p_i=p$ for all $i\in\{1,\dots,m\}$. First, we prove that $p_1=p$. Assume, by contradiction, that $p_1\ne p$. Then, we can write $p=p'\mu$, where $p'\ne p_1'$. Consequently, we have $p_1'^*p=p_1'^*p'\mu=0$. Let $\cdot$ denote the action of $L_K(E)$ on $V_{[\mu]}^f$. Since $p_1'$ does not involve $e_1$, it follows that $\sigma(p_1'^*)=p_1'^*$, which implies that $p_1'^*u=p_1'^*\cdot u$ for all $u\in V_{[\mu]}^f$. Moreover, as $p_i'$'s are distinct, we have $p_1'^*p_i=p_1'^*p_i'\mu=0$ for all $i\ne 1$. It follows that 
	$$
	0=\varphi(p_1'^*p)=p_1'^*\varphi( p)=p_1'^*\sum_{i=1}^{m}\lambda_i p_i=\lambda_1p_1'^*p_1=\mu\ne0.
	$$
	This contradiction shows that $p_1=p$. Similarly, we also obtain that $p_i=p$ for all $i\in\{2,\dots m\}$. We have thus proven that for any $p\in[\mu]$, 
	$$
	\varphi(p)=\lambda_p p, \text{ for some } \lambda_p\in K'.
	$$
	Next, we claim that the scalar $\lambda_p$ does not depend on $p$. Take another $q\in [\mu]$ and write $\varphi(q)=\lambda_qq$. Then, we can write $p=p'\mu$ and $q=q'\mu$ where $p'$ and $q'$ do not involve $e_1$. Then
	$$
	\varphi(\lambda_p^{-1}\mu)=\varphi(\lambda_p^{-1}p'^*p)=\lambda_p^{-1}p'^*\varphi(p)=\lambda_p^{-1}p'^*\lambda_pp=\mu.
	$$
	Similarly, we get $\varphi(\lambda_q^{-1}\mu)=\mu$. It follows that $\varphi(\lambda_p^{-1}\mu)=\varphi(\lambda_q^{-1}\mu)$. Because $V_{[\mu]}^f$ be the simple $L_K(E)$-module, we conclude that $\varphi$ is injective, which implies that $\lambda_p=\lambda_q$, and the claim is shown. Let $\lambda_{\varphi}\in K'$ denote the common value. So $\varphi(u)=\lambda_{\varphi} u$, for all $u\in V_{[\mu]}^f$. Finally, the map ${\rm End}_{L_K(E)}(V_{[\mu]}^f)\to K'$, defined by $\varphi\mapsto \lambda_{\varphi}$, is the required isomorphism.
\end{proof}

(D): \textit{Ara-Rangaswamy module $\mathbf{S}_{v\infty}$ defined by an infinite emitter $v$}: Let $v$ be an infinite emitter in $E$. In \cite{Pa_rangas_2015}, Ara and Rangaswamy defined $\mathbf{S}_{v\infty}$ to be the $K$-vector space whose basis is the set $B=\{\mu: \mu \hbox{ is a path in } E \hbox{ with } r(\mu)=v\}$; that is,
$$
\mathbf{S}_{v\infty}=\bigoplus_{\substack{\mu\in{\rm Path}(E), \\ r(\mu)=v}}K\mu.
$$
For each vertex $u$ and each edge $e$ in $E$, they define linear transformations $P_u$, $S_e$ and $S_{e^*}$ on $\mathbf{S}_{v\infty}$ as follows. For any $\mu\in B$, set

\begin{align*}
	P_u(\mu) =    \begin{cases} 
		\mu        & \text{if } u=s(\mu) \\
		0        & \text{otherwise }
	\end{cases} ,                        					
	& \hspace{1cm} S_e(\mu) =    \begin{cases} 
		e\mu      & \text{if } r(e)=s(\mu) \\
		0        & \text{otherwise }
	\end{cases},\\
	S_{e^*}(\mu) =    \begin{cases} 
		\mu'       & \text{if } \mu=e\mu' \\
		0        & \text{otherwise }
	\end{cases},										
	& \hspace{1cm} S_{e^*}(v) =   0.
\end{align*}
The mapping $\phi: L_K(E) \to {\rm End}_K(\mathbf{S}_{v\infty})$, which sends $u$ to $P_u$, $e$ to $S_e$ and $e^*$ to $S_{e^*}$, is an algebra homomorphism. Therefore, $\mathbf{S}_{v\infty}$ can be viewed as a left module over $L_K(E)$ via $\phi$.  In \cite{Pa_rangas_2015}, Rangaswamy proved that $\mathbf{S}_{v\infty}$ is a simple left $L_K(E)$-module. For further use, we need to compute ${\rm End}_K(\mathbf{S}_{v\infty})$.
\begin{lemma}\label{lemma_2.3}
	If  $v$ is an infinite emitter in a graph $E$, then $
	{\rm End}_{L_K(E)}\left( \mathbf{S}_{v\infty}\right) \cong K.$
\end{lemma}
\begin{proof}
	Assume that $\mu, \eta\in B$. It is easy to check that $\mu^*\eta=\delta_{\mu\eta}v$.  Consider a nonzero element $\varphi\in {\rm End}_{L_K(E)}(S_{v\infty})$. Then, $\varphi$ is an automorphism because $S_{v\infty}$ is a simple $L_K(E)$-module. For an arbitrary $\mu\in B$, let
	$$
	\varphi(\mu)=\sum_{i=1}^\ell k_i\eta_i,
	$$
	where $k_i\in K\backslash\{0\}$ and $p_i$'s are pairwise distinct. We claim that $\eta_i=\mu$ for all $i\in\{1,\dots,l\}$. First, we prove that $\eta_1=\mu$. Assume, by contradiction that $\eta_1\ne \mu$, then we get $\eta_1^*\mu=0$. So, we get 
	$$
	0=\varphi(\eta_1^*\mu)=\eta_1^*\varphi(\mu)=\eta_1^*\sum_{i=1}^\ell k_i\eta_i,=k_1\eta_1^*\eta_1=k_1v\ne 0.
	$$
	This contradiction implies that $\eta_1=\mu$. Similarly, we also get that $\eta_i=\mu$ for all $i\geq 2$. Hence, $\varphi(\mu)=k_{\mu}\mu$ for some $k_{\mu}\in K$. We claim that $k_{\mu}$ depends only on $\varphi$, not on $\mu$. Indeed, for any $\eta\in B$, we can write $\varphi(\eta)=k_{\eta}\eta$ for some $k_{\eta}\in K$. Since $r(\mu)=r(\eta)=v$, we have
	$$
	k_{\mu}\mu=\varphi(\mu)=\varphi(\mu v)=\varphi(\mu\eta^*\eta)=\mu\eta^*\varphi(\eta)=\mu\eta^*(k_{\eta}\eta)=k_{\eta}\mu(\eta^*\eta)=k_{\eta}\mu,
	$$
	and it follows that $k_{\mu}=k_{\eta}$. The claim is shown. Therefore, we conclude that for every nonzero element $\varphi\in {\rm End}_{L_K(E)}(S_{v\infty})$, there exists a unique element, say, $\lambda_{\varphi}\in K$ such that $\varphi(\mu)=\lambda_{\varphi}\mu$ for all $\mu\in B$. If we set additionally  $\lambda_0=0$, then, clearly, the assignment $\varphi\mapsto \lambda_{\varphi}$ defines an isomorphism between ${\rm End}_{L_K(E)}(S_{v\infty})$ and $K$. 
\end{proof}

Now, we are ready to prove the following proposition  which is the main result of this section.

\begin{proposition}\label{proposition_2.4}
	Let $E$ be an arbitrary graph, $K$ an arbitrary field, and $P$ any primitive ideal of $L_K(E)$. Then, there exists a simple left module $V$ which is one of the types $V_{[\mu]}$, $\mathbf{N}_w$, $V_{[\mu]}^f$, and $\mathbf{S}_{v\infty}$ such that ${\rm Ann}_{L_k(E)}(V)=P$ and  ${\rm End}_{L_K(E)}(V)$ is a field containing $K$. 
\end{proposition}
\begin{proof}
	The existence of $V$ follows from \cite[Propositions 2.6 and 2.7]{Pa_rangas_2015} and \cite[Theorem~ 3.9]{Pa_ara_rangas_2014}. For the second assertion, we note that, according to Lemmas \ref{lemma_2.1}, \ref{lemma_2.2} and \ref{lemma_2.3}, we conclude that ${\rm End}_{L_K(E)}(V)$ is isomorphic to $K$ or $K[x,x^{-1}]/(f(x))$, where $f(x)$ is an irreducible polynomial in $K[x, x^{-1}]$.
\end{proof}
\section{Non-cyclic free subgroups in Leavitt path algebras}\label{free}

Let $L_K(E)$ be a Leavitt path algebra over a field $K$ of characteristic $0$. The main aim of this section is to prove that the multiplicative group $L_K(E)^\times$  contains a non-cyclic free subgroup provided that $L_K(E)$ is  noncommutative. 

\begin{lemma}\label{Lemma on primitive ideal}
	Let $L_K(E)$ be a noncommutative Leavitt path algebra. Then, there exists a primitive ideal $P$ of $L_K(E)$ such that $L_K(E)/P$ is a noncommutative primitive ring.
\end{lemma}
\begin{proof}
	Recall that for any graph $E$ and any field $K$, the Leavitt path algebra $L_K(E)$ is a semiprimitive ring (see \cite[Proposition 2.3.2]{Bo_abrams_2017}).	
	Let $\{P_i\mid i\in I\}$ be the set of all primitive ideals in $L_K(E)$.	Then $\bigcap_{i\in I}P_i=0$ and there exists an injective ring homomorphism
	$$\varepsilon: L_K(E)\longrightarrow\prod_{i\in I}(L_K(E)/P_i).$$
	
	Since $L_K(E)$ is noncommutative, we conclude that $\prod_{i\in I}(L_K(E)/P_i)$ is noncommutative too. It
	follows that there exists $i_0\in I$ such that $L_K(E)/P_{i_0}$ is noncommutative. By setting $P=P_{i_0}$, the proof of the lemma is now complete.
\end{proof}

To proceed further, we need one result by Sanov on the existence of non-cyclic free subgroup in the general linear group $\mathrm{GL}_2(\mathbb{Z})$ of degree $2$ over the ring $\mathbb{Z}$ of integers. In fact, in 1947, Sanov \cite{Pa_Sanov_1947} proved that the subgroup $H=\left\langle \begin{pmatrix} 1 & 0  \\ 2 & 1\\ \end{pmatrix},\\
\begin{pmatrix} 1 & 2 \\ 0 & 1 \\\end{pmatrix}
\right\rangle $ is a free subgroup of the group ${\rm GL}_2(\mathbb{Z})$. Based on Sanov's result, we can state the following lemma, which is convenient for our use in the present paper.
\begin{lemma}[{\cite[Example 25, p.26]{Bo_Harpe_2000}}]\label{lemma_3.5} Let $K$ be a field of characteristic $0$. Then, the group $\left\langle \begin{pmatrix} 1 & 0  \\ 2 & 1\\ \end{pmatrix},\\
	\begin{pmatrix} 1 & 2 \\ 0 & 1 \\\end{pmatrix}
	\right\rangle $ is a free subgroup of the linear group ${\rm GL}_2(K)$.
\end{lemma}

This result, although seemingly unrelated to Leavitt path algebras, provides a crucial tool for our investigation. By leveraging a suitable analogy between ${\rm GL}_2(K)$ and certain properties of $L_K(E)$, we will be able to utilize this lemma to demonstrate that $L_K(E)^{\times}$ always contains a non-cyclic free subgroup. Furthermore, the lemma will also contribute to describing these free subgroups in detail.

For the sake of completeness, we include the proof of the following elementary fact, as we are unaware of a reference. This result will be used many times throughout the paper without further reference.

\begin{lemma}\label{lemma_3.1}
	Let $R$ and $S$ be unital rings, and let $\varphi: R\to S$ be a ring surjective  homomorphism. Assume that $\left\langle x, y\right\rangle $ is a non-cyclic free subgroup of $S^\times$ generated by $x$ and $y$. If $a\in \varphi^{-1}(x)$ and $b\in \varphi^{-1}(y)$ such that $a, b \in R^\times$, then $\left\langle a, b\right\rangle $ is a non-cyclic free subgroup of $R^\times$ and $\left\langle a, b\right\rangle \cong \left\langle x, y\right\rangle$.
\end{lemma}
\begin{proof}
	It is clear that the induced homomorphism $\bar{\varphi}:= \varphi |_{\langle a,b \rangle}$ of $\varphi$ on $\langle a,b \rangle $ is a surjective group homomorphism such that $\bar{\varphi}(a)=x$ and $\bar{\varphi}(b)=y$. By the universal property of the free group $\langle x,y \rangle $, there exists a group homomorphism $f: \langle x,y \rangle\to \langle a,b \rangle $ such that $f(x)=a$ and $f(y)=b$. This guarantees that $\bar{\varphi}$  is an isomorphism with $(\bar{\varphi})^{-1}=f$, and so $ \langle x,y \rangle\cong \langle a,b \rangle $.
\end{proof}

We are now in the position to prove the main result of this section.

\begin{theorem}\label{free subgroup}
	Let $L_K(E)$ be a noncommutative Leavitt path algebra over a field $K$ of characteristic $0$.  Then, the multiplicative group $L_K(E)^\times$ contains a non-cyclic free subgroup. 
\end{theorem}
\begin{proof}
		In view of Lemma \ref{Lemma on primitive ideal}, there exists a primitive ideal $P$ of $L_K(E)$ such that $L_K(E)/P$ is noncommutative. Also, Proposition \ref{proposition_2.4} implies that there is a left simple $L_K(E)$-module $V$ such that  $P={\rm Ann}_{L_K(E)}(V)$ and that $k:={\rm End}_{L_{K}(E)}(V)$ is a field containing $K$. Let $\varphi: L_K(E)\to L_K(E)/P$ be the natural epimorphism. Then, via $\varphi$, $V$ became $L_K(E)/P$-module.  A routine checking shows that $V$  is a left faithful simple $L_K(E)/P$-module, and  ${\rm End}_{L_K(E)/P}(V)={\rm End}_{L_{K}(E)}(V)=k$. By Density Theorem for Primitive ring (see e.g., \cite[Theorem (11.19)]{Bo_lam_2001}), there are two possible cases:
	
	\bigskip
	
	\textit{\textbf{Case 1}: $L_K(E)/P$ is artinian.} If we set $n=\dim_kV$, then $n<\infty$ and $L_K(E)/P\cong {\rm M}_n(k)$. Since $L_K(E)/P$ is noncommutative, we must have $n>1$. Since $K\subseteq k$, and $\mathrm{char}(K)=0$, by Lemma \ref{lemma_3.5}, ${\rm GL}_2(k)$ contains a non-cyclic free subgroup in two generators, say $H=\langle x, y\rangle$. Embedding ${\rm GL}_2(k)$ into ${\rm GL}_n(k)$ by obvious ways, we see that $H$ can be considered as a free subgroup of ${\rm GL}_n(k)$. It can be shown that we can choose specific elements $a\in \varphi^{-1}(x)$ and $b\in \varphi^{-1}(y)$ such that both $a$ and $b$ belong to the multiplicative group $L_K(E)^\times$. The construction of these specific elements and their connection to the base graph $E$ will be detailed in the next section (see Theorems \ref{type I}, \ref{type II}, and \ref{describing}).  Therefore, we may apply Lemma \ref{lemma_3.1} to conclude that $ \left\langle a,b\right\rangle $ is a non-cyclic free subgroup of   $L_K(E)^\times$. 
	
	\bigskip
	
	\textit{\textbf{Case 2}: $L_K(E)/P$ is not artinian.} Again, according to \cite[Theorem 11.19]{Bo_lam_2001}, for each $n>1$, there exists a subring $\Lambda_n$ of $L_K(E)/P$ and a surjective ring homomorphism $f: \Lambda_n\to{\rm M}_n(k)$. Let $x$, $y$ be the generators of a free subgroup of ${\rm GL}_n(k)$ given in Lemma \ref{lemma_3.5}. Since $\varphi$ and $f$ are surjective, there exist $a, b\in L_K(E)$ such that $\varphi(f(a))=x$ and $\varphi(f(b))=y$. We may choose $a$ and $b$ such that $a, b\in L_K(E)^\times$ (again, this fact will be shown in the next section by Theorems \ref{type I}, \ref{type II}, and \ref{describing}). Finally, it follows from Lemma  \ref{lemma_3.1} that $ \left\langle a,b\right\rangle $ is a non-cyclic free subgroup of   $L_K(E)^\times$. 
\end{proof}

\section{The generators of  non-cyclic free subgroups}\label{generators}

Theorem \ref{free subgroup} in the precedent section shows the existence of non-cyclic free subgroups in the multiplicative group of a noncommutative Leavitt path algebra. The existence of such free subgroups was shown based on the existence of such a primitive ideal $P$ of $L_K(E)$ that $L_K(E)/P$ is noncommutative (such an ideal $P$ exists in view of Lemma \ref{Lemma on primitive ideal}). However, describing the generators of these free subgroups in detail seems to be a challenging task.  This section is devoted to examining this problem and providing a comprehensive description of these generators. The next theorem, which follows immediately from \cite[Theorem 3.12(iii) and Theorem 4.3]{Pa_rangas_2013},  classifies primitive ideals of $L_K(E)$ by three types.
\begin{theorem}\label{th:4.1}
	For a given ideal $P$ of $L_K(E)$, set $H=P\cap E^0$. Then, $P$ is primitive if and only if $P$ is one of the following types:
	\begin{enumerate}[font=\normalfont]
		\item[I:] $P$ is a  graded ideal of the form $I(H,B_H\backslash\{w\})$ for some $w\in B_H$ and $M(w)=E^0\backslash H$.
		\item[II:] $P$ is a graded ideal of the form $I(H,B_H)$ such that $E^0\backslash H$ satisfies the {\rm MT-3} condition and the countable separation property, and $E/(H, B_H)$ satisfies the condition ${\rm (L)}$.
		\item[III:] $P=I(H, B_H, f(c))$, where $c$ is an exclusive cycle based at a vertex $u$, $E^0\backslash H=M(u)$, and $f(x)$ is an irreducible polynomial in $K[x, x^{-1}]$.
	\end{enumerate}
\end{theorem}

\begin{remark}\label{rem:4.2}
	Primitive ideals of ${\rm type \; I}$ are always non-zero because $B_H\ne\varnothing$ and so $H\ne\varnothing$, and ones of ${\rm type \; III}$ are also non-zero as $f(c)\ne0$. On the other hand, the primitive ideals of ${\rm type \; II}$ are possibly zero because it is possible that $H=B_H=\varnothing$, and so $L_K(E)$ is a primitive ring in this case.
\end{remark}
Taking any primitive ideal $P$ of $L_K(E)$ such that $L_K(E)/P$ is noncommutative (such an ideal exists according to Lemma \ref{Lemma on primitive ideal}), we shall determine the type of $P$ according to Theorem \ref{th:4.1}, and then we describe non-cyclic free subgroups of $L_K(E)^\times$ with their generators by the type of $P$. 

\subsection{Free subgroups determined by the primitive ideals of type I}

In this subsection we describe free subgroups of $L_K(E)$ by generators in case when the ideal $P$ is of type I.

\begin{lemma}\label{lem:4.3}
	Assume that $E$ contains a sink $w$ and an edge $f$ such that $r(f)=w$. Then $\left\langle 1+2f^*, 1+2f\right\rangle $ is a non-cyclic free subgroup of $L_K(E)^{\times}$.
\end{lemma}
\begin{proof}
	Let $V=Kw\oplus Kf$ be the subspace of the $K$-space $\mathbf{N}_w$ defined by the sink $w$ as in Section 2(B).  Then, this is a non-zero $K$-space with a $K$-basis $\{w, f\}$ and $\dim_KV=2$. Put
	$$
	R=\{x\in L_K(E): xV\subseteq V\} \text{ and } I=\{x\in L_K(E): x V=0\}, 
	$$ 
	then $R$ is a subring of $L_K(E)$ and $I$ is an ideal of $R$, so $R/I$ can be viewed as a $K$-space. Now, we are going to seek a $K$-basis for $R/I$. Take $\lambda, \gamma\in\mathrm{Path}(E)$ such that $r(\gamma)= r(\lambda)$, and $v := a_1w+a_2f\in V$,  where $a_1, a_2\in K$. We first make the following computations
	$$
	a_1\gamma\lambda^* w =   \begin{cases} 
		a_1\gamma      & \text{ if }  \lambda=w,\\
		0      & \text{ otherwise}.
	\end{cases}
	$$
	
	$$
	a_2\gamma\lambda^* f =    \begin{cases} 
		a_2\gamma f      & \text{ if }  \lambda=s(f),\\
		a_2\gamma      & \text{ if }  \lambda=f,\\
		0      & \text{ otherwise}.
	\end{cases}
	$$
	It follows that 
	$$
	\gamma\lambda^* v = a_1\gamma\lambda^* w+a_2\gamma\lambda^*f 
	=  \begin{cases} 
		a_1\gamma      & \text{ if }  \lambda=w,\\
		a_2\gamma f      & \text{ if }  \lambda=s(f),\\
		a_2\gamma      & \text{ if }  \lambda=f,\\
		0      & \text{ otherwise}.
	\end{cases}\eqno(1)
	$$
	Now, take $x\in L_K(E)$ and write 
	$$
	x=k_1\gamma_1\lambda_1^*+k_2\gamma_2\lambda_2^*+\dots+k_n\gamma_n\lambda_n^*,
	$$
	where $k_i\in K$ and $ \gamma_i, \lambda_i \in{\rm Path}(E) $ with $r(\gamma_i)=r(\lambda_i)$ for all $1\leq i \leq n$. 
	In view of (1), $x \in R$ if and only if either $x v=0 $, or else $\lambda_i\in \{s(f), w, f\}$ and, as $xv\in V$, $\gamma_i \in \{s(f), w, f\}$. Equivalently, $x \in R$ if and only if either $x v=0 $ or $\gamma_i\lambda_i^*\in \{s(f), w, f, f^*, ff^*\}$.  Since, we have 
	$$
	(s(f)-ff^*)v=(s(f)-ff^*)(a_1w+a_2f)=-a_2(f-f)=0.
	$$
	Because $v$ was taken arbitrarily, we conclude that $(s(f)-ff^*)V=0$, from which it follows that $s(f)-ff^*\in I$. This means that 	
	$$
	R/I={\rm span}_K\left\lbrace s(f)+I,  f+I, f^*+I, w+I\right\rbrace.
	$$ 
Now, we will check that these elements are linearly independent over $K$. Let $ a,b,c,d\in K$ be such that 
	$$
	a(s(f)+I)+b(f+I)+c(f^*+I)+d(w+I)=0+I.
	$$
	Then $as(f)+bf+cf^*+dw\in I$, which means that $(as(f)+bf+cf^*+dw)V=0$. So,
	$$
	\begin{cases} 
		(as(f)+bf+cf^*+dw)f=0,\\
		(as(f)+bf+cf^*+dw)w=0.
	\end{cases}
	$$
	As $s(f)\ne w= r(f)$, we get that 
	$$
	\begin{cases} 
		af+cw=0,\\
		dw+bf=0.
	\end{cases}
	$$
	As $w, f$ are linearly independent over $K$, we have $a=b=c=d=0$. Thus, $s(f)+I, f+I, f^*+I,  w+I\in R/I$ are linearly independent over $K$. Therefore, the set $\{s(f)+I,  f^*+I, r(f)+I, f+I \}$ is a $K$-basis for $R/I$. Accordingly, for every element $x\in R$, the element $x+I$ can be written uniquely in the form
	$$a(s(f)+I) + b(f+I) + c(f^*+I) +d(w+I), \text{ where } a, b, c, d\in K.$$
Let us consider the mapping
$$\varphi: R\longrightarrow \mathrm{M}_2(K),$$
defined by $\varphi(x)=\begin{pmatrix} a & b  \\ c & d\\ \end{pmatrix}$ for any $x\in R$. Clearly, $\varphi$ is a surjective ring homomorphism with $\ker \varphi =I$.

Consider the following matrices in ${\rm M}_2(K)$: 
	\begin{align*}
		A             =    \begin{pmatrix} 1 & 0  \\ 2 & 1\\ \end{pmatrix},\;\;\;
		B             =    \begin{pmatrix} 1 & 2  \\ 0 & 1\\ \end{pmatrix}.
	\end{align*}
	It is easy to see that $\varphi(1+2f^*)=A$ and $\varphi(1+2f)=B$. According to Lemma \ref{lemma_3.5}, $\left\langle A,B\right\rangle $ is a non-cyclic free subgroup of ${\rm GL}_2(K)$. Since $(2f)^2=0$ and $(2f^*)^2=0$, it follows that $1+2f^*$ and $1+2f$ are invertible in $L_K(E)$ with 
	\begin{align*}
		&	(1+2f^*)^{-1}     =   1-2 f^*,\\
		&	(1+2f )^{-1}     =   1-2 f.
	\end{align*}
	Since $\varphi: R\to {\rm M}_2(K)$ is surjective, we may apply Lemma \ref{lemma_3.1} to conclude that $\left\langle 1+2f^*, 1+2f\right\rangle $ is a  non-cyclic free subgroup of $R^\times \subseteq L_K(E)^\times$.
\end{proof}

\begin{example}
	Let $E_T$ be the Toeplitz graph:
	
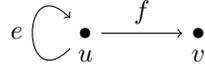
\begin{center}
		\begin{tikzpicture}
			%\node at (-2,0) (1) {$E_T:$};
			\node at (0,0) (0) {$\bullet$};
			\node at (0,-0.3) {$u$};
			\node at (1.5,0) (2) {$\bullet$};
			\node at (1.5,-0.3) {$v$};
			
			\draw [->] (0) to [out=225,in=135,looseness=6] node[left] {$e$} (0);
			\draw [->] (0) to node[above] {$f$} (2);
		\end{tikzpicture}
	\end{center}
	Let $\mathscr{T}_K=L_K(E_T)$  be the Toeplitz $K$-algebra $\mathscr{T}_K$. It follows from Lemma \ref{lem:4.3} that $\left\langle 1+2f^*, 1+2f\right\rangle $ is a  non-cyclic free subgroup of $\mathscr{T}_K^{\times}$.
\end{example}

\begin{lemma}\label{lem:4.4}
	Let $(H,S)$ be an admissible pair in a graph $E$,  and $\varphi: L_K(E)\to L_K(E/(H,S))$ the epimorphism given in Remark \ref{rem:1.2}. Assume that $f\in E^1$ with $w=r(f)\in B_H\backslash S$. Then we have 
	$$\varphi(s(f))=s(f'), \;\; \varphi(w^H)=w', \;\;  \varphi(fw^H)=f',  \;\; \varphi(w^Hf^*)=f'^*,$$
	where $w'$ and $f'$ are the corresponding vertex and edge in the quotient graph $E/(H,S)$ of $w$ and $f$ respectively.
\end{lemma}
\begin{proof}
	Definition \ref{def:1.1} confirms that $s(f')=s(f)$, so by Remark \ref{rem:1.2}, we have  $\varphi(s(f))=s(f')$. For the proof of the remaining relation, we recall that $$w^H=w-\sum_{\{s(e)=w,\;r(e)\not\in H\}}ee^*.$$
	For the sake of simplicity, we set $\Delta=\sum_{\{s(e)=w,\;r(e)\not\in H\}}ee^*$. First, we claim that $\varphi(\Delta)=w\in  (E/(H,S))^0$. Let $\{e_1, e_2,\dots,e_n\}$ be an enumeration of the elements of the set $\{e\in E^1| s(e)=w,r(e)\not\in H\}$. (Note that since $w$ is a breaking vertex of $H$, this set is non-empty and finite.) If $r(e_i)\not\in B_H\backslash S$ for all $1\leq i\leq n$, then, by Remark~\ref{rem:1.2}, we have $\varphi(e_i)=e_i$ and $\varphi(e_i^*)=e_i^*$ for all $i$. In this case, the set of all edges in $E/(H,S)$ having the common source $w$ is $\{e_1, e_2,\dots,e_n\}$. Therefore, the relation (CK2) in $L_K(E/(H,S))$ implies that $\varphi(\Delta)=\sum_{i=1}^{n}e_ie_i^*=w\in  (E/(H,S))^0$ Now, assume that some of the $r(e_i)$'s belong to $B_H\backslash S$ and the remaining of them do not. Without loss of the generality, we may assume that $r(e_1),\dots, r(e_m) \in B_H\backslash S$ and $r(e_{m+1}),\dots, r(e_n)\not\in B_H\backslash S$, where $1\leq m\le n$. It follows from Remark  \ref{rem:1.2} that $\varphi(e_i)=e_i+e_i'$, $\varphi(e_i^*)=e_i^*+e_i'^*$ for $1\leq i\leq m$ and $\varphi(e_i)=e_i$, $\varphi(e_i^*)=e_i^*$ for $m+1\leq i\leq n$. Consequently, we have 
	$$
	\varphi(\Delta)=\sum_{i=1}^{m}(e_i+e_i')(e_i^*+e_i'^*)+\sum_{i=m+1}^{n}ee_i^*.
	$$
	For each $1\leq i\leq m$, in $L_K(E/(H,S))$, we have the following relation $$(e_i+e_i')(e_i^*+e_i'^*)=e_ie_i^*+e_ie_i'^*+e_i'e_i^*+e_i'e_i'^*.$$
	Because $r(e')$ is a sink in $E/(H,S)$, we have $e_ie_i'^*=e_i'e_i^*=0$, from which it follows that $(e_i+e_i')(e_i^*+e_i'^*)=e_ie_i^*+e_i'e_i'^*$. This implies that 
	$$ \varphi(\Delta)=\sum_{i=1}^{m}e_ie_i^*+\sum_{i=1}^{m}e_i'e_i'^*+\sum_{i=m+1}^{n}ee_i^*=\sum_{i=1}^{n}e_ie_i^*+\sum_{i=1}^{m}e_i'e_i'^*.
	$$
	Observe that the set of all  edges in $E/(H,S)$ having the common source $w$ is $$\{e_1, e_2,\dots,e_n\} \cup \{e_1',\dots,e_m'\}.$$ Again, the relation (CK2) in $L_K(E/(H,S))$ implies that $ \varphi(\Delta)=w\in (E/(H,S))^0$, and the claim is shown. Using the fact that $\varphi(\Delta)=w$, we now easily make the following computations:
	\begin{align*}
		&	\varphi(w^H)\hspace*{0.3cm}			= 	\varphi(w)-\varphi(\Delta) = w+w'-w=w'.\\
		&	\varphi(fw^H)\hspace*{0.1cm}			= 	\varphi(f)\varphi(w)-\varphi(f)\varphi(\Delta)\\&\hspace*{1.4cm} = (f+f')(w+w')-(f+f')w=f+f'-f=f'.\\	
		&	\varphi(w^Hf^*)		  =   \varphi(w)\varphi(f^*)-\varphi(\Delta)\varphi(f^*) = (w+w')(f^*+f'^*)-w(f^*+f'^*)\\&\hspace*{1.4cm}=f^*+f'^*-f^*=f'^*.	
	\end{align*}
	Therefore, the proof is now complete.
\end{proof}

Now, we are ready to prove the main theorem of this subsection, describing the generators of a free subgroup of $L_K(E)^\times$ in case $P$ is of type I.
\begin{theorem}\label{type I}
	Let $P$ be a primitive ideal of $L_K(E)$ such that $L_K(E)/P$ is noncommutative. 
	Assume that $P$ is of type {\rm I}. Let $w$ be the breaking vertex determining $P$ as in Theorem \ref{th:4.1}. Then, there exists $f\in E^1$ with $r(f)=w$ such that $ \left\langle 1+2 w^Hf^* , 1+ 2fw^H \right\rangle $ is a non-cyclic free subgroup of $L_K(E)^{\times}$.
\end{theorem}
\begin{proof}
	Take a vertex $w\in B_H$ as in the part I of Theorem \ref{th:4.1}. That is, $w\in B_H$ is  a vertex such that $P=I(H, B_H\backslash\{w\})$ and $M(w)=E^0\backslash H$. Let $F=E/(H,B_H\backslash\{w\})$.  Then,  there is an epimorphism $\varphi: L_K(E) \to L_K(F)$ with $\ker\varphi = P$. The quotient graph $F$ is described as follows:
	$$ 
	\begin{aligned}
		& F^0= (E^0\backslash H) \cup \{w'\};\\
		& F^1=\{e\in E^1: r(e)\not\in H\}\cup \{f':  r(f)=w\},
	\end{aligned}
	$$
	and $r$, $s$ are extended to $F^0$ by setting $s(f')=s(f)$ and $r(f')=w'$ for all $f\in E^1$ with $r(f)=w$.  If $r^{-1}(w)=\varnothing$, then $F^0$ would be the single point $w'$ which means that $L_K(F)$ is commutative, a contradiction. This guarantees that $r^{-1}(w)\ne\varnothing$. Now, take $f$ such that $r(f)=w$. Set $f'=\varphi(f)$. Then $f'$ is an edge in $F^1$ satisfying $r(f')=w'$. Because $w'$ is a sink in $F$, we conclude that $s(f')\ne r(f')={w'}$. In view of Lemma \ref{lem:4.4},  $\varphi(s(f))=s(f')$, $\varphi(w^H)=w'$, $\varphi(fw^H)=f'$ and $\varphi(w^Hf^*)=f'^*$.
	This implies that 
	\begin{align*}
		& 	\varphi\left( 1+2 w^Hf^* \right)  = 1+2 f'^*,\\
		&	\varphi\left( 1+2fw^H \right) = 1+2 f'.
	\end{align*}
	It follows from Lemma  \ref{lem:4.3} that $\left\langle 1+2 f'^*,1+2 f'\right\rangle $ is a non-cyclic free subgroup of $L_K(F)^\times$. Moreover, because $(2w^Hf^*)^2=0$ and $(2fw^H)^2=0$, we get that the inverses of $1+2 w^Hf^*$ and $1+2 fw^H $ are $1-w^Hf^*$ and $1-fw^H$  respectively. Therefore, Lemma \ref{lemma_3.1} implies that $\left\langle 1+2 w^Hf^*,1+2fw^H \right\rangle $ is a non-cyclic free subgroup of $L_K(E)^\times$.
\end{proof}

The following simple example serves an illustration for Theorem \ref{type I}.

\begin{example}
	Let $E$ be the below graph: 
	
	
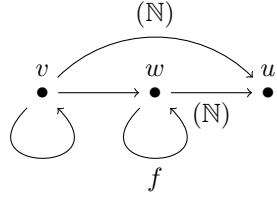
\begin{center}
		\begin{tikzpicture}
			\node at (0,0) (w) {$\bullet$};
			\node at (0,0.3) {$w$};
			\node at (1.5,0) (u) {$\bullet$};
			\node at (1.5,0.3) {$u$};
			\node at (-1.5,0) (v) {$\bullet$};
			\node at (-1.5,0.3) {$v$};
			
			\draw [->] (v) to node[left] {$ $} (w);
			\draw [->] (v) to [out=45,in=135,looseness=1] node[above] {$(\mathbb{N})$} (u);
			\draw [->] (w) to node[below] {$(\mathbb{N})$} (u);
			\draw [->] (w) to [out=-135,in=-45,looseness=8] node[below] {$f$} (w);
			\draw [->] (v) to [out=-135,in=-45,looseness=8] node[below] {$ $} (v);
		\end{tikzpicture}
	\end{center}
	Put $H=\{u\}$. Then $H$ is a hereditary and saturated and $B_H=\{v, w\}$. Then, $P:=I(H, B_H\backslash \{w\})$ is a primitive graded ideal of type I  of $L_K(E)$. It is clear that $w^H=w-ff^*$. Therefore, the following elements will generate a non-cyclic free subgroup in $L_K(E)^\times$:
	$$
	1+(w-ff^*)f^* \text{ and } 1+ f(w-ff^*).
	$$	
\end{example}

\subsection{Free subgroups determined by the primitive ideals of type II}
In this subsection, we study the case when $P$ is of type II.
\begin{lemma}\label{lem:4.6}
	Let $E$ be a graph which contains two distinct tail-equivalent infinite paths $\mu'$ and $\eta'$. Then, there exist two infinite paths of the form $\mu$ and $f\mu$ which are tail-equivalent to $\mu'$ and $\eta'$, where $f\in E^1$ such that $r(f)=s(\mu)$.
\end{lemma}

\begin{proof}
	Since $\mu'$ and $\eta'$ are tail-equivalent infinite paths in $E$, there exist $m$, $n\geq 0$ which are minimal such that $\tau_{> m}(\mu')=\tau_{> n}(\eta')$. Write $\mu'=s(\mu')e_1\cdots e_m\tau_{> m}(\mu')$ and $\eta'=s(\eta')f_1\cdots f_n\tau_{> n}(\eta')$, where $e_i, f_j\in E^1$. Since $\mu'\ne \eta'$, $m$ and $n$ cannot be equal to $0$ at the same time. Without loss of generality, we may assume that $n\geq 1$. Clearly, the paths $\mu=\tau_{> m}(\mu')=\tau_{> n}(\eta')$ and $\eta=f_n\mu$ are infinite of the desired form.
\end{proof}
The following lemma serves as a tool for determining a non-cyclic free subgroup in $L_K(E)$ generated by an infinite path. 

\begin{lemma}\label{lem:4.7}
	Assume that $E$ contains an edge $f$ which is an initial edge of an infinite path $q=fp$ such that $s(f)\ne r(f)=s(p)$. Then $\left\langle 1+2 f^*, 1+2 f\right\rangle $ is a non-cyclic free subgroup of $L_K(E)^{\times}$.
\end{lemma}

\begin{proof}
	
	Let $V_{[q]}$ be the $K$-space defined by $q$ as in Section 2(A). Consider the $K$-subspace $V=Kp\oplus Kq$ of $V_{[q]}$. Then, $V$ is a $K$-space with a basis $\{p, q\}$, so $\dim_KV=2$. Put
	$$
	R=\{x\in L_K(E): xV\subseteq V\} \text{ and } I=\{x\in L_K(E): x V=0\}. 
	$$ 
	Then, $R$ is a subring of $L_K(E)$ and $I$ is an ideal of $R$. Therefore, we can view $R/I$ as a $K$-space in an obvious way. Next, we are going to find a $K$-basis for $R/I$. To do this, take  $\gamma, \lambda\in{\rm Path}(E)$ with $r(\gamma)=r(\lambda)$, and $v = a_1p+a_2q\in V$, where $a_1, a_2\in K$. We have
	$$
	a_1\gamma\lambda^* p = \begin{cases} 
		a_1\gamma\eta      & \text{ if }   p=\lambda\eta  \text{ for some }\eta\in {\rm Path}(E), \\
		0      & \text{ otherwise}.
	\end{cases}
	$$
	$$
	a_2\gamma\lambda^* q = \begin{cases} 
		a_2\gamma\beta   &  \text{ if }  q=\lambda\beta \text{ for some }  \beta\in{\rm Path}(E),\\
		0      & \text{ otherwise}.
	\end{cases}
	$$
	Since $s(p)\ne s(q)$, the cases $p=\lambda\eta$ and $q=\lambda\beta$ cannot happen at the same time. This implies that
	$$
	\gamma\lambda^*v= \gamma\lambda^* (a_1p+a_2q)=    \begin{cases} 
		a_1\gamma\eta      & \text{ if }  p=\lambda\eta,\\
		a_2\gamma\beta   &  \text{ if }  q=\lambda\beta,\\
		0      & \text{ otherwise}.
	\end{cases}
	$$
	Therefore, $\gamma\lambda^*v\in V$ if and only if either $\gamma\lambda^*v=0$ or one of the following cases occurs:
	
	\medskip 
	
	\textit{\textbf{Case 1}: $p=\lambda\eta$ and $\gamma\eta\in V$}. The first assertion implies that $\lambda=\tau_{\le n}(p)$ for some $n\geq0$. Since $\gamma\eta\in V$, it follows that $\gamma\eta=p$ or $\gamma\eta=q$. We consider two subcases:
	
	\medskip 
	
	\textit{Subcase 1.1}: $\gamma\eta=p$. In this case, the computation above shows that $\gamma\lambda^*v=a_1p$. Moreover, we have $\gamma\eta=p=\lambda\eta$, which implies that $\lambda^*\gamma \eta=\eta$. It follows 
    that $\lambda^*\gamma \ne0$, and so $\lambda = \gamma$ by (CK1). Therefore, $$\gamma\lambda^*\in S_1:=\{\tau_{\le n}(p)\tau_{\le n}(p)^*: n\ge~0\}.$$
    
	\medskip 
	
	\textit{Subcase 1.2}: $\gamma\eta=q$. In this case, the computation above shows that $\gamma\lambda^*v=a_1q$. Because $\lambda\eta=p$ and $\gamma\eta=q=fp$, we conclude that $f\lambda\eta=\gamma\eta$, and so $f\lambda=\gamma$. Therefore, we get $\gamma=f\tau_{\le n}(p)$, from which it follows that 
	$$\gamma\lambda^*\in S_2:=\{f\tau_{\le n}(p)\tau_{\le n}(p)^*: n\ge~0\}.$$
	
	\medskip 
	
	%\textit{\textbf{Case 2}: $q=\lambda\beta$ and $\gamma\beta\in V$}. The fact $q=\lambda\beta$ implies that $\lambda= \tau_{\le m}(q)$ for some $m\geq0$. As $\gamma\beta\in V$, either $\gamma\beta=p$ or $\gamma\beta=q$. Repeating the earlier arguments we conclude that either $\gamma\lambda^*v=a_2p$ and $$\gamma\lambda^*\in\{\tau_{\le m}(p)\tau_{\le m}(p)^*f^*: m\ge~0\},$$ or else $\gamma\lambda^*v=a_2q$ and $$\gamma\lambda^*\in\{\tau_{\le m}(p)\tau_{\le m}(p)^*: m\ge~0\}.$$

	\textit{\textbf{Case 2}: $q=\lambda\beta$ and $\gamma\beta\in V$}. The fact $q=\lambda\beta$ implies that $\lambda= \tau_{\le m}(q)$ for some $m\geq0$. As $\gamma\beta\in V$, either $\gamma\beta=p$ or $\gamma\beta=q$. Consider two subcases: 
	
	\medskip 
	
	\textit{Subcase 2.1}: $\gamma\beta=p$.  In this case, the computation above shows that $\gamma\lambda^*v=a_2p$. As $\gamma\beta=p$ and $\lambda\beta=q=fp$, we get $f\gamma\beta=\lambda\beta$, yielding $\lambda=f\gamma$. This implies that $m\geq1$ and $\gamma=\tau_{\le m-1}(p)$, and so $$\gamma\lambda^*\in S_3:=\{\tau_{\le m-1}(p)\tau_{\le m-1}(p)^*f^*: m\ge~1\}.$$
	
	\medskip 
	
	\textit{Subcase 2.2}: $\gamma\beta=q$. In this case, the computation above shows that $\gamma\lambda^*v=a_2q$. As $\gamma\beta=q= \lambda\beta$, we get $\lambda=\beta$. Thus, 
	$$\gamma\lambda^*\in S_4:=\{\tau_{\le m}(q)\tau_{\le m}(q)^*: m\ge~0\}.$$
	
	\medskip 
	
	The computations we just made allow us to conclude that $\gamma\lambda^* V\subseteq V$ if and only if $\gamma\lambda^* V=0$ (which means $\gamma\lambda^*\in I$), or else $\gamma\lambda^*\in\bigcup_{j=1}^4S_j$. This means that each element of $R/I$ has the form $
	\sum_{i =1}^k k_i\gamma_i\lambda_i^*+I$ where  $k_i\in K$ and $\gamma_i\lambda_i^*\in\bigcup_{j=1}^4S_j$ for all $1\leq i\leq k$. 

Now, for any monomial $\gamma\lambda^*$, and arbitrary element $v=ap+bq\in V$ ($a, b\in K$), there are four cases to consider. 	
 
(1) Assume that $\gamma\lambda^*\in S_1$. Then, we have
 	\begin{eqnarray*}
 		\gamma\lambda^*v
 		&=&\tau_{\le n}(p)\tau_{\le n}(p)^*(ap+bq)\\
 		&=&a\tau_{\le n}(p)\tau_{\le n}(p)^*\tau_{\le n}(p)\tau_{>n}(p)+0\\
 		&=&ap.
 	\end{eqnarray*}
 	On the other hand, we also have
 	\begin{eqnarray*}
 		r(f)v
 		&=&r(f)(ap+bq)\\
 		&=&ap+br(f)q\\
 		&=&ap+0=ap.
 	\end{eqnarray*}
 	Hence, $\left(\gamma\lambda^*-r(f)\right) v=0$, which means that $\gamma\lambda^*+I=r(f)+I$.
 
(2) Assume that $\gamma\lambda^*\in S_2$. Then,	
 	\begin{eqnarray*}
 		\gamma\lambda^*v
 		&=&f\tau_{\le n}(p)\tau_{\le n}(p)^*(ap+bq)\\
 		&=&afp=aq.
 	\end{eqnarray*}
 	On the other hand, we have
 	\begin{eqnarray*}
 		fv
 		&=&f(ap+bq)\\
 		&=&afp+0=aq.
 	\end{eqnarray*}
 	Hence, it follows that $\gamma\lambda^*+I=f+I$.
 
(3) Assume that $\gamma\lambda^*\in S_3$. Then,
 	\begin{eqnarray*}
 		\gamma\lambda^*v
 		&=&\tau_{\le m-1}(p)\tau_{\le m-1}(p)^*f^*(ap+bq)\\
 		&=&0+b\tau_{\le m-1}(p)\tau_{\le m-1}(p)^*f^*q\\
 		&=&b\tau_{\le m-1}(p)\tau_{\le m-1}(p)^*p\\
 		&=&b\tau_{\le m-1}(p)\tau_{\le m-1}(p)^*\tau_{\le m-1}(p)\tau_{> m-1}(p)\\
 		&=&bp.
 	\end{eqnarray*}
 	On the other hand, we have $f^*v=f^*(ap+bq)=bp$. Hence, $\gamma\lambda^*+I=f^*+I$.	
 
(4) Assume that $\gamma\lambda^*\in S_4$. Then,
 	\begin{eqnarray*}
 		\tau_{\le m}(q)\tau_{\le m}(q)^*v
 		&=&\tau_{\le m}(q)\tau_{\le m}(q)^*(ap+bq)\\
 		&=&0+b\tau_{\le m}(q)\tau_{\le m}(q)^*\tau_{\le m}(q)\tau_{> m}(q)=bq.
 	\end{eqnarray*}
 	On the other hand, we have $s(f)v=s(f)(ap+bq)=bq$. Hence, $\gamma\lambda^*+I=s(f)+I$. The arguments above allows us to conclude that $R/I$ is described precisely as follows:
	$$
	R/I=\left\lbrace as(f)+bf+cf^*+dr(f)+I\;|\; a, b, c, d\in K\right\rbrace .
	$$
	The same argument used in the proof of Lemma \ref{lem:4.3} shows that  
	$$\{s(f)+I, f+I, f^*+I, r(f)+I\}$$
	is a $K$-basis of $R/I$. Now, by the same argument as in the proof of Lemma \ref{lem:4.3}, we conclude that $\left\langle 1+2f^*, 1+2f\right\rangle $ is a  non-cyclic free subgroup of $R^\times \subseteq L_K(E)^\times$.
\end{proof}

We are now in the position to present the main theorem of this subsection.
\begin{theorem}\label{type II}
	Let $P$ be a primitive ideal of $L_K(E)$ such that $L_K(E)/P$ is noncommutative. 
	Assume that $P$ is of type {\rm II}.
	Then $E$ contains an edge $f$ which is an initial edge of an infinite path or ends in a sink such that $ \left\langle 1+2f^*,  1+2 f\right\rangle $ is a non-cyclic free subgroup of $L_K(E)^{\times}$.
\end{theorem}
\begin{proof}
	Put $H=P\cap E^0$. If we set $F=E/(H,B_H)$, then in view of Remark \ref{rem:1.2}, there is an epimorphism  $\varphi: L_K(E) \to L_K(F)$ with $\ker\varphi = P$. In this case, the quotient graph $F$ is described as follows:
	$$ 
	F^0= E^0\backslash H \text{ and } F^1=\{e\in E^1: r(e)\not\in H\}, \eqno(1)
	$$
	and $r$, $s$ are restricted to $F^1$. According to the proof of \cite[Theorem 3.9]{Pa_ara_rangas_2014}(part (iii)), either there exists a unique sink $w$ in $F$ such that $F^0=M(w)$, or there is an infinite path $\mu_0$ in $F$ such that $F^0=M(\mu_0)$. Hence, there are two possible cases:
	
	\smallskip 
	
	\textit{\textbf{Case 1}: $F^0$ contains a unique sink $w$ such that $F^0=M(w)$.} Consider $w$ as a vertex in $E^1$, we see that $w$ is either a sink or an infinite emitter in $E$ such that $r(s^{-1}(w)\subseteq H$. If $r^{-1}(w)=\varnothing$ then $F$ would be the single vertex $w$ which says that $L_K(F)$ is commutative, a contradiction. Therefore $r^{-1}(w)\ne\varnothing$, and thus we may take $f\in F^1\subseteq E^1$ with $r(f)=w$.  Now, by Remark \ref{rem:1.2}, we have $\varphi(f)=f$ and $\varphi(w)=w$. Since $w$ is a sink in $F$, we deduce that $s(f)\ne r(f)=w$ in $F$. Therefore, we may apply Lemma  \ref{lem:4.3} to conclude that $\left\langle 1+2 f^*, 1+2 f\right\rangle $ is a non-cyclic free subgroup of $L_K(F)^\times$. Also, it is easy to check that $1+2 f^*$ and $1+2 f$ are invertible in $L_K(E)^{\times}$. As $ 1+2 f^*$ and $ 1+2 f$ are fixed by $\varphi$, we conclude that $\left\langle 1+2 f^*, 1+2 f\right\rangle $ is a non-cyclic free subgroup of $(L_K(E))^\times$. 
	
	\smallskip 
	
	\textit{\textbf{Case 2}: There is an infinite path $\mu_0$ in $F$ such that $F^0=M(\mu_0)$}: Because $F$ is a subgraph of $E$, we can consider $\mu_0$ as an infinite path in $E$ which does not end in an exclusive cycle in $E$ such that $E^0=M(\mu_0)$. Let $V_{[\mu_0]}$ be the faithful $L_K(F)$-simple module determined by the infinite path $\mu_0$; that is, 
	$$
	V_{[\mu_0]}=\bigoplus_{\eta\in [\mu_0]}K\eta,
	$$
	a $K$-vector space whose basis consists of all infinite paths $\eta$ which are tail- equivalent to $\mu_0$. According to Lemma \ref{lemma_2.1}, we conclude that ${\rm End}_{L_K(F)}(V_{[\mu_0]})\cong K$. If $L_K(F)$ is artinian, then \cite[Theorem 11.19]{Bo_lam_2001} implies that $L_K(F)\cong {\rm M}_n(K)$, where $n=\dim_K(V_{[\mu_0]})$ which is the number of infinite paths in $F$ tail-equivalent to $\mu_0$. Since $L_K(F)$ is noncommutative, it follows that $n\geq 2$. If $L_K(F)$ is non-artinian, then $\dim_K(V_{[\mu_0]})=~\infty$. In either case, we have $\dim_K(V_{[\mu_0]})\geq 2$, and so there exist in $F$ two infinite distinct paths $\mu'$ and $\eta'$ which are tail-equivalent to $\mu_0$. According to Lemma \ref{lem:4.6}, we conclude that there exist two infinite paths $\mu$ and $f\mu$, where $f\in F^1\subseteq E^1$ with $s(f)\ne r(f)=s(\mu)$. According to Lemma \ref{lem:4.7}, we obtain that $\left\langle 1+2 f^*, 1+2 f\right\rangle $ is a non-cyclic free subgroup of $(L_K(F))^\times$. The same argument as in previous case shows that $\left\langle 1+2 f^*, 1+2 f\right\rangle $ is also a non-cyclic free subgroup of $(L_K(E))^\times$.
\end{proof}
\begin{example}
	Let $E$ be the following graph:
	
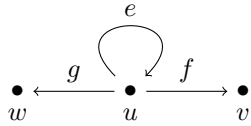
\begin{center}
		\begin{tikzpicture}
			\node at (3,0) (v) {$\bullet$};
			\node at (3,-0.3) {$v$};
			\node at (0,0) (w) {$\bullet$};
			\node at (0,-0.3) {$w$};
			\node at (1.5,0) (u) {$\bullet$};
			\node at (1.5,-0.3) {$u$};
			
			\draw [->] (u) to node[above] {$f$} (v);
			\draw [->] (u) to node[above] {$g$} (w);
			\draw [->] (u) to [out=135,in=45,looseness=8] node[above] {$e$} (u);
		\end{tikzpicture}
	\end{center}
	Then  $H=\{w\}$  is a hereditary and saturated subset of $E^0$ and $P:=I(H)$ is a primitive ideal of type II (note that $B_H=\varnothing$). Also $E^0\backslash H=\{u,v\}=M(v)$. According to Theorem \ref{type II}, we conclude that $\left\langle 1+2 f^*, 1+2 f\right\rangle$ is a non-cyclic free subgroup of  $L_K(E)^\times$.
\end{example}
\begin{example}
	Consider the graphs $E$ and $F$ as follows:
	
	
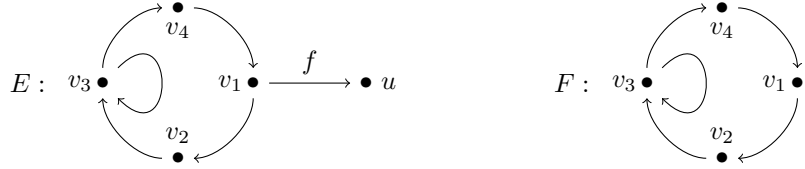
\begin{center}  \begin{figure}[!htb]   \begin{minipage}{0.5\textwidth}    \centering    \begin{tikzpicture}     \node at (-2,0) (o) {$E:$};     \node at (1,0) (v1) {$\bullet$};     \node at (0.7,0) {$v_1$};     \node at (-1,0) (v3) {$\bullet$};     \node at (-1.3,0) {$v_3$};     \node at (0,1) (v4) {$\bullet$};     \node at (0,0.7) {$v_4$};     \node at (0,-1) (v2) {$\bullet$};     \node at (0,-0.7) {$v_2$};     \node at (2.5,0) (u) {$\bullet$};     \node at (2.8,0) {$u$};          \draw [->] (v1) to node[above] {$f$} (u);     \draw [->] (v3) to [out=90,in=180,looseness=0.8] node[above] {$ $} (v4);     \draw [->] (v4) to [out=0,in=90,looseness=0.8] node[above] {$ $} (v1);     \draw [->] (v1) to [out=-90,in=0,looseness=0.8] node[above] {$ $} (v2);     \draw [->] (v2) to [out=180,in=-90,looseness=0.8] node[above] {$ $} (v3);     \draw [->] (v3) to [out=45,in=-45,looseness=7] node[above] {$ $} (v3);    \end{tikzpicture}   \end{minipage}\hfill   \begin{minipage}{0.5\textwidth}    \centering    \begin{tikzpicture}     \node at (-2,0) (o) {$F:$};     \node at (1,0) (v1) {$\bullet$};     \node at (0.7,0) {$v_1$};     \node at (-1,0) (v3) {$\bullet$};     \node at (-1.3,0) {$v_3$};     \node at (0,1) (v4) {$\bullet$};     \node at (0,0.7) {$v_4$};     \node at (0,-1) (v2) {$\bullet$};     \node at (0,-0.7) {$v_2$};         \draw [->] (v3) to [out=90,in=180,looseness=0.8] node[above] {$ $} (v4);     \draw [->] (v4) to [out=0,in=90,looseness=0.8] node[above] {$ $} (v1);     \draw [->] (v1) to [out=-90,in=0,looseness=0.8] node[above] {$ $} (v2);     \draw [->] (v2) to [out=180,in=-90,looseness=0.8] node[above] {$ $} (v3);     \draw [->] (v3) to [out=45,in=-45,looseness=7] node[above] {$ $} (v3);    \end{tikzpicture}   \end{minipage}  \end{figure} 
	\end{center}
	
	Then $H=\{u\}$ is a hereditary and saturated subset of $E^0$ and $E/H=F$. It can be checked that $I(H)$ is a primitive ideal of type II and and $L_K(E)/I(H)\cong L_K(F)$ which is noncommutative. Let $c=e_1e_2e_3e_4$ be a cycle in $E$ and $\mu=c^\infty$. Then $\mu$ is an infinite path in $E$ with $E^0\backslash H=M(\mu)$. Also, the edge $e_4$ satisfies the condition $v_4=s(e_4)\ne r(e_4)=v_1=s(\mu)$. According to Theorem \ref{type II}, we conclude that $\left\langle 1+2 e_4^*, 1+2 e_4\right\rangle$ is a non-cyclic free subgroup of  $L_K(E)^\times$.
	
	We note in passing that there are many choices of infinite paths in $E$ which produce different non-cyclic free subgroups. For example, we may choose $\eta=e^\infty$ and the edge $e_2$ that satisfies $s(e_2)\ne r(e_2)=s(\eta)$.
\end{example}
As we have noted in Remark \ref{rem:4.2}, being a primitive ideal of type II, it is possible that $P=0$. If it is the case, then $L_K(E)$ is a primitive algebra. As it was proven in \cite[Thereom 3.10]{Pa_ara_rangas_2014}, the Leavitt path algebra $L_K(E)$ is primitive if and only if E contains either a sink $w$ or an infinite path $\mu$ that does not end in a cycle without exits, such that $E^0=M(w)$ or $M(\mu)$. Building on this result, the following corollary of Theorem \ref{type II} provides a description of the free subgroups in the multiplicative group of a primitive Leavitt path algebra.

\begin{corollary}\label{cor:4.9}
	If $L_K(E)$ is a primitive, then $E$ contains an $f$ which is an initial edge of an infinite path $\mu$ that does not end in a cycle without exits, or ends in a sink such that $ \left\langle 1+2f^*,  1+2 f\right\rangle $ is a non-cyclic free subgroup of $L_K(E)^{\times}$.
\end{corollary}

\begin{proof} 
	In the proof of Theorem \ref{type II}, we see that the quotient graph $F$ contains either a sink or an infinite path, namely $w$ and $\mu$ respectively, and an edge $f$ with $r(f)=w$ or $s(\mu)$. Under these conditions we have proved that $L_K(F)^{\times}$ contains a non-cyclic free subgroup. Since the graph $E$ in the present corollary satisfies all conditions as for the graph $F$ defined in the proof of Theorem \ref{type II}, the proof of the existence of non-cyclic free subgroups in $L_K(F)$ is applied similarly for $L_K(E)$ here. Hence, we can conclude that the proof of the corollary is complete. 
\end{proof}

\subsection{Free subgroups determined by the primitive  ideals of type III}
This subsection is devoted to the study the case when $P$ is of type III.
\begin{theorem}\label{th:4.10}
	Let $P$ be a primitive ideal of $L_K(E)$ such that $L_K(E)/P$ is noncommutative. 
	Assume that $P$ is of type {\rm III}.
	Then $E$ contains an edge $f$ which is the initial edge of an infinite path $\mu$ tail-equivalent to $c$ such that $ \left\langle 1+2f^*,  1+2 f\right\rangle $ is a non-cyclic free subgroup of $L_K(E)^{\times}$.
\end{theorem}

\begin{proof}
	Put $F=E/(H, B_H)$. By Remark \ref{rem:1.2}, we get that $I(H, B_H)$  is a graded ideal of $L_K(E)$, and there is an epimorphism $\varphi: L_K(E) \to L_K(F)$ with $\ker\varphi = I(H,B_H)$. It is clear that $F$ contains $c$ which is an exclusive cycle based at $u$ and $F^0=M(u)$. Let $V_{[c]}$ be the faithful simple $L_K(F)$-simple module determined by $c$. The same arguments used in Case 2 of Theorem \ref{type II} shows that $\dim_K(V_{[c]})\geq 2$, so there exist in $F$ two infinite distinct paths $\mu'$ and $\eta'$ which are tail-equivalent to $c$. According to Lemma \ref{lem:4.6}, we conclude that there exist two infinite paths $\mu$ and $f\mu$ where $f\in F^1\subseteq E^1$ with $s(f)\ne r(f)=s(\mu)$. Therefore the result follows from Lemma~\ref{lem:4.7}.
\end{proof}
\begin{example}
	Let $E$ be the following graph:
	
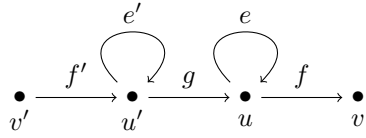
\begin{center}
		\begin{tikzpicture}
			%\node at (-2.5,0) {$E:$};
			\node at (-1.5,0) (v') {$\bullet$};
			\node at (-1.5,-0.3) {$v'$};
			\node at (3,0) (v) {$\bullet$};
			\node at (3,-0.3) {$v$};
			\node at (0,0) (u') {$\bullet$};
			\node at (0,-0.3) {$u'$};
			\node at (1.5,0) (u) {$\bullet$};
			\node at (1.5,-0.3) {$u$};
						
			\draw [->] (u) to node[above] {$f$} (v);
			\draw [->] (v') to node[above] {$f'$} (u');
			\draw [->] (u') to node[above] {$g$} (u);
			\draw [->] (u) to [out=135,in=45,looseness=8] node[above] {$e$} (u);
			\draw [->] (u') to [out=135,in=45,looseness=8] node[above] {$e'$} (u');
		\end{tikzpicture}
	\end{center}
	As we have showed, the set $H=\{v\}$ is hereditary and saturated, and for each irreducible polynomial $f(x)\in K[x, x^{-1}]$, the ideal $P:=I(H,f(e))$ is primitive of type III and $L_K(E)/P$ is noncommutative. If we set $p=e^\infty$, then $\mu$ is an infinite path with $E^0\backslash H=M(\mu)$. Also, it is clear that the edge $g$ satisfies the condition $u'=s(g)\ne r(g)=u=s(\mu)$. It follows that $ \left\langle 1+2g^*,  1+2 g\right\rangle $ is a non-cyclic free subgroup of $L_K(E)^{\times}$.
\end{example}
Now, combining the main results we get in three subsections above, we obtain the following theorem describing the generators of free subgroups of $L_K(E)^\times$ as desired. 
\begin{theorem}\label{describing}
	Let $E$ be a graph and $K$ a field. Then, there exists an edge $f\in E^1$ such that one of the following two conditions occurs:
	
	\begin{enumerate}[font=\normalfont]	
		\item $f$ is either the initial edge of an infinite path, or it ends at a sink, such that  $ \left\langle 1+2f^*,  1+2 f\right\rangle $ is a non-cyclic free subgroup of $L_K(E)^{\times}$.
		\item $f$ ends at an infinite emitter $w$ which must be a breaking vertex of a hereditary and saturated subset $H$ of $E^0$, such that $ \left\langle 1+2 w^Hf^*,  1+ 2fw^H\right\rangle $ is a non-cyclic free subgroup of $L_K(E)^{\times}$ where $w^H=w-\sum_{\{\substack{s(e)=w,\;r(e)\not\in H}\}}ee^*$.
	\end{enumerate}
\end{theorem}
\begin{proof}
	According to Lemma \ref{Lemma on primitive ideal}, there always exists a primitive ideal $P$ of $L_K(E)$ such that $L_K(E)/P$ is a noncommutative primitive ring. If $P$ is a primitive ideal of type II or III, then the part (1) of Theorem \ref{describing} follows from Theorem \ref{type II}, Corollary \ref{cor:4.9}, and \ref{th:4.10}. If $P$ is of type I, then the part (2) of Theorem \ref{describing} follows from Theorem \ref{type I}. Hence, the theorem is proved.
\end{proof}
%\begin{remark}  Theorem~\ref{free subgroup} also admits an analogous version for fields of positive characteristic. However, in this case, the base field must not be locally finite (a field where every finitely generated subfield is finite). The generators of the free subgroup in this case exhibit additional complexities compared to those presented in Theorem~\ref{free subgroup}. We leave the details of this extension to the interested reader.
%\end{remark}

{\noindent\textbf{Funding} }This research is funded by the Vietnam Ministry of Education and Training under grant number B2024-CTT-02.

\bigskip 

{\noindent\textbf{Conflict of Interest.} }The authors have no conflict of interest to declare that are relevant to this article.

\end{document}